\documentclass{amsart}
\usepackage{amssymb,euscript,tikz,units}
\usepackage[colorlinks,citecolor=blue,linkcolor=red]{hyperref}
\usepackage{verbatim}
\usepackage[nameinlink]{cleveref}
\usepackage{colonequals}
\usepackage{tikz}
\usepackage[titletoc]{appendix}
\usepackage{enumerate}
\usepackage[all]{xy}

\renewcommand{\H}{\mathbb{H}}

\def\area{\mathop{\rm Area}}

\DeclareMathOperator{\Def}{\overset{\text{def}}{=}}

\theoremstyle{plain}
\newtheorem{theorem}{Theorem}
\newtheorem{corollary}[theorem]{Corollary}
\newtheorem{proposition}[theorem]{Proposition}
\newtheorem{lemma}[theorem]{Lemma}
\newtheorem{subl}[theorem]{SubLemma}

\newtheorem{remark}[theorem]{Remark}

\newcommand{\be}{\begin{equation}}
\newcommand{\ene}{\end{equation}}
\newcommand{\br}{\begin{remark}}
\newcommand{\er}{\end{remark}}
\newcommand{\bl}{\begin{lem}}
\newcommand{\el}{\end{lem}}
\newcommand{\bcor}{\begin{cor}}
\newcommand{\ecor}{\end{cor}}
\newcommand{\bpro}{\begin{pro}}
\newcommand{\epro}{\end{pro}}
\newcommand{\ben}{\begin{enumerate}}
\newcommand{\een}{\end{enumerate}}
\newcommand{\bp}{\begin{proof}}
\newcommand{\ep}{\end{proof}}
\newcommand{\bpo}{\begin{pro}}
\newcommand{\epo}{\end{pro}}
\newcommand{\beq}{\begin{equation*}}
\newcommand{\eeq}{\end{equation*}}
\newcommand{\bear}{\begin{eqnarray}}
\newcommand{\eear}{\end{eqnarray}}
\newcommand{\beqar}{\begin{eqnarray*}}
\newcommand{\eeqar}{\end{eqnarray*}}
\newcommand{\bt}{\begin{theorem}}
\newcommand{\et}{\end{theorem}}
\newcommand{\bex}{\begin{excer}}
\newcommand{\eex}{\end{excer}}

\theoremstyle{definition}

\theoremstyle{remark}

\newtheorem*{rem*}{Remark}
\newtheorem*{exam*}{Example}
\newtheorem*{exams*}{Examples}
\newtheorem*{thm*}{\bf Theorem}
\newtheorem*{que*}{Question}
\newtheorem*{Def*}{Definition}
\newtheorem*{Conj*}{\bf Conjecture}
\newtheorem*{Assum*}{\bf Assumption}
\begin{document}
\title[Cheeger constants]{The Cheeger constants of random Belyi surfaces}

\author{Yang Shen and Yunhui Wu}
\address{Yau Mathematical Sciences Center and Department of Mathematics, Tsinghua University, Beijing, China}
\email[(Y.~S.)]{shen-y19@mails.tsinghua.edu.cn}
\email[(Y.~W.)]{yunhui\_wu@tsinghua.edu.cn}

\maketitle

\begin{abstract}
Brooks and Makover developed a combinatorial model of random hyperbolic surfaces by gluing certain hyperbolic ideal triangles. In this paper we show that for any $\epsilon>0$, as the number of ideal triangles goes to infinity, a generic hyperbolic surface in Brooks-Makover's model has Cheeger constant less than $\frac{3}{2\pi}+\epsilon$.
\end{abstract}

\section{Introduction}\label{sec-int}
Given a closed hyperbolic surface $X_g$ of genus $g\geq 2$, the \emph{Cheeger constant} $h(X_g)$ of $X_g$ is defined as
\[h(X_g)\Def\inf_{E\subset X_g} \frac{\ell(E)}{\min\left\{ \area(A),\area(B)\right\}}\]
where $E$ runs over all one-dimensional subsets of $X_g$ dividing $X_g$ into two disjoint components $A$ and $B$, and $\ell(E)$ is the length of $E$.
The Cheeger constant $h(X_g)$ can bound the first eigenvalue $\lambda_1(X_g)$ of $X_g$ from both sides. Actually the well-known Cheeger-Buser \cite{Cheeger70,Buser-ine} inequality says that $$\frac{h^2(X_g)}{4}\leq \lambda_1(X_g)\leq 2 h(X_g)+10 h^2(X_g).$$ In particular, $\lambda_1(X_g) \to 0$ if and only if $h(X_g)\to 0$. For large genus, by Cheng \cite{Cheng75} it is known that
\be \label{uug}
\limsup \limits_{g\to \infty}  h(X_g)\leq 1
\ene for any sequence of hyperbolic surfaces $\{X_g\}$ of genus $g$.

Brooks and Makover \cite{BM04} developed a combinatorial model of a random closed surface with large genus by first gluing together $2n$ $(n>0)$ copies of an ideal hyperbolic triangle and then taking its conformal compactification, where the gluing scheme is given by a random trivalent graph. It is known  (e.g., see \cite[Lemma 2.1]{MR2271484}) that such constructions give all the so-called Belyi surfaces which are dense in the space of all Riemann surfaces in some sense (e.g., see \cite{Be79}). In this model, certain classical geometric quantities were studied for large $n$. For example: as the parameter $n\to \infty$, they showed \cite[Theorem 2.3]{BM04} that the expected value of the genus of a random hyperbolic surface roughly behaves like $\frac{n}{2}$; they also showed \cite[Theorem 2.2]{BM04} that as $n\to \infty$, a generic hyperbolic surface in their model has Cheeger constant greater than $C_0$ where $C_0>0$ is an implicit uniform constant. For the other direction, in light of \eqref{uug} it is natural to ask
\begin{que*}  \label{Q-CC}
Is there an $\epsilon_0>0$ so that as $n\to \infty$, a generic hyperbolic surface in Brooks-Makover's model has Cheeger constant less than $1-\epsilon_0$? If yes, similar as in \cite{BZ02}, can $\epsilon_0$ be chosen to be greater than $\frac{1}{2}$?
\end{que*}

Now we briefly recall the terminologies in Brooks-Makover's model of random hyperbolic surfaces \cite{BM04}. Set
\begin{align*}
\mathcal{F}^\star_n=\left\{(\Gamma,\mathcal{O});\begin{matrix} &\Gamma\text{ is a 3-regular graph with $2n$ vertices}\\ &\text{and }\mathcal{O}\text{ is an orientation on }\Gamma\end{matrix}\right\}.
\end{align*}
As in \cite{BM04}, each pair $(\Gamma,\mathcal{O})\in \mathcal{F}^\star_n$ gives two Riemann surfaces $S^O(\Gamma,\mathcal{O})$ and $ S^C(\Gamma,\mathcal{O})$ where $S^O(\Gamma,\mathcal{O})$ is an open Riemann surface constructed by gluing $2n$ ideal hyperbolic triangles in a certain way, and  $S^C(\Gamma,\mathcal{O})$ is the conformal compactification of  $S^O(\Gamma,\mathcal{O})$. Let $\textnormal{Prob}_n$ be the uniform measure on $\mathcal{F}^\star_n$ introduced by Bollob\'as \cite{Bollobas-iso}, which one may also see \cite[Section 5]{BM04} for more details. In this paper we prove the following result which in particular gives a positive answer to the question above. More precisely,
\begin{theorem}\label{mt-i}
Let $(\Gamma,\mathcal{O})$ be a random element of $\mathcal{F}^\star_n$. Then for any $\epsilon>0$,
$$\lim \limits_{n\to \infty}\textnormal{Prob}_n\left\{(\Gamma,\mathcal{O})\in \mathcal{F}^\star_n; \ h\left(S^C(\Gamma,\mathcal{O})\right)<\frac{3}{2\pi}+\epsilon\right\}=1.$$
\end{theorem}

\begin{rem*}
\ben
\item Brooks and Zuk in \cite{BZ02} showed that as the covering degree goes to $\infty$, congruence covers of the moduli surface $\mathbb{H}/\mathrm{SL}(2,\mathbb{Z})$ have Cheeger constants less than $0.4402$. This roughly says that the analogue of the famous Selberg's $\frac{1}{4}$ eigenvalue conjecture in the context of the Cheeger constant is false (see \cite[Page 52]{BZ02}).

\item For the non-compact case, recently we showed in \cite{SW22} that as the genus $g$ goes to infinity, a Weil-Petersson cusped hyperbolic surface has arbitrarily small Cheeger constant provided that the number of cusps grows significantly faster than $g^{\frac{1}{2}}$.

\item By Buser's inequality \cite{Buser-ine}, a uniform spectral gap also yields a uniform positive lower bound for Cheeger constant. For this line, one may see \emph{e.g.} \cite{Sel65, GJ78, Iwa89, LRS95, KS03} for congruence covers of the moduli surface $\mathbb{H}/\mathrm{SL}(2,\mathbb{Z})$; see \emph{e.g.} \cite{MN20, MNP20, HM21} for random covering surfaces; and see \emph{e.g.} \cite{Mirz13, WX22-GAFA, LW21, Hide21} for Weil-Petersson random surfaces.
\een
\end{rem*}

We remark here that after this paper was submitted, very recently Budzinski, Curien and Petri showed in \cite{BCP22} that
\be
\limsup \limits_{g\to \infty}  h(X_g)\leq \frac{2}{\pi} \nonumber
\ene for any sequence of hyperbolic surfaces $\{X_g\}$ of genus $g$. This solved \cite[Problem $10.5$]{Wright-tour} due to Wright. It would be \emph{interesting} to know that whether this upper bound $\frac{2}{\pi}$ can be replaced by $\frac{1}{2}$.

\subsection*{Strategy on the proof of Theorem \ref{mt-i}.} We briefly introduce the idea on the proof of Theorem \ref{mt-i} here. By definition of the Cheeger constant, it suffices to show that for a generic surface $S^C(\Gamma,\mathcal{O})$, there exists a colletion of curves whose union separates $S^C(\Gamma,\mathcal{O})$ into two parts such that the ratio quantity in $h(S^C(\Gamma,\mathcal{O}))$ can be bounded from above by the desired upper bound in Theorem \ref{mt-i}. By the construction of hyperbolic surfaces in \cite{BM04}, we know that $$S^C(\Gamma,\mathcal{O})=\left(\bigcup\limits_{i\in\mathcal{I}}D_i\right)\bigcup\text{\{$2n$ small triangles\}}$$ where each $D_i$ is a horoball in $S^O(\Gamma,\mathcal{O})$ together with its infinity point, and each small triangle has three boundary curves of lengths $\leq 1$ (for example see Figure \ref{fig:15}). For large $n>0$ and a generic surface $S^C(\Gamma,\mathcal{O})$, we first split as $\bigcup\limits_{i\in\mathcal{I}}D_i=\left(\bigcup\limits_{i\in\mathcal{I}_1}D_i\right)\bigcup \left(\bigcup\limits_{i\in\mathcal{I}_2}D_i\right)$ where each $D_i$ with $i\in \mathcal{I}$ contains a hyperbolic disk of large radius (see Section \ref{sec-bounds}). Next we construct certain simple curves $\{\eta_i^j\}$ of lengths at most $O(\log n)$ separating each $D_i$ $(i\in \mathcal{I}_1)$ into certain subdomains $\{D(\eta_i^j,\eta_i^k)\}$, each of which has area at most $O\left(\frac{n}{(\log n)^2}\right)$ (see Lemma \ref{l-decom}). Rewrite the decomposition as
$$S^C(\Gamma,\mathcal{O})=\bigcup\limits_{i\in\mathcal{I}_1}
\left(\bigcup\limits_{j=1}^{k_i}D_{ij}\right)\bigcup\left(\bigcup
\limits_{i\in\mathcal{I}_2}D_i\right)\bigcup\{\text{$2n$ small triangles}\}.$$

\noindent Then one can make divisions of $S^C(\Gamma,\mathcal{O})$ as follows: fix two symbols $\mathcal{A},\mathcal{B}$ and define the so-called compatible mapping $J:S^C(\Gamma,\mathcal{O})\to \{\mathcal{A},\mathcal{B}\}$ such that 
\ben
\item on the interior of each piece of the decomposition of $S^C(\Gamma,\mathcal{O})$ above, $J$ is constant either $\mathcal{A}$ or $\mathcal{B}$; 
\item for three horocycle segments in any small triangle, there exists at most one of them such that it (if exists) is contained in the boundary $\partial \mathcal{A}(J)=\partial \mathcal{B}(J)$ where
\be \label{def=divi}
\mathcal{A}(J)=\bigcup\limits_{\Omega\in S^C(\Gamma,\mathcal{O}),\ J(\Omega)=\mathcal{A}}\Omega\text{ \ and \ }\mathcal{B}(J)=\bigcup\limits_{\Omega\in S^C(\Gamma,\mathcal{O}),\ J(\Omega)=\mathcal{B}}\Omega. \nonumber
\ene
\een
Each map $J$ induces a division $(\mathcal{A}(J),\mathcal{B}(J))$ of $S^C(\Gamma,\mathcal{O})$ (for example see Figure \ref{fig:15}). Define $X_n$ to be the set of all compatible mappings from $S^C(\Gamma,\mathcal{O})$ to $\{\mathcal{A},\mathcal{B}\}$, which is a finite set endowed with a uniform probability measure. Now we view $\ell_C(\partial \mathcal{A}(J))$ and $\min\left\{\textnormal{Area}_C(\mathcal{A}(J)), \textnormal{Area}_C(\mathcal{B}(J))\right\}$ as two random variables on $X_n$. Here the randomness only comes from the mapping $J$ and the surface $S^C(\Gamma,\mathcal{O})$ is fixed. With the help of Lemma \ref{l-decom} we show that for almost a generic $(\Gamma,\mathcal{O})\in \mathcal{F}_n^*$ defined in \eqref{def-Fnc}, the expected values satisfy  

$$\mathbb{E}\left[\ell_C(\partial \mathcal{A}(J))\right]\leq \frac{3n}{2}+o(n) \quad \text{(see Lemma \ref{l-exp-1})}$$  
and
$$\mathbb{E}\left[\min\left\{\textnormal{Area}_C(\mathcal{A}(J)), \textnormal{Area}_C(\mathcal{B}(J))\right\}\right]\geq (1-2\delta)^2 n \pi-o(n) \quad \text{(see Lemma \ref{l-exp-3})}$$ where $\delta \in (0, \frac{1}{2})$ is arbitrary. Recall that for any $J$,
\[h(S^C(\Gamma,\mathcal{O}))\leq \frac{\ell(\partial \mathcal{A}(J))}{\min\{\area(\mathcal{A}(J)),\area(\mathcal{B}(J))\}}. \]
Then Theorem \ref{mt-i} follows by letting $n\to \infty$ and $\delta \to 0$ since
\[h(S^C(\Gamma,\mathcal{O}))\leq \frac{\mathbb{E}\left[\ell_C(\partial \mathcal{A}(J))\right]}{\mathbb{E}\left[\min\left\{\textnormal{Area}_C(\mathcal{A}(J)), \textnormal{Area}_C(\mathcal{B}(J))\right\}\right]}\leq \frac{3}{2\pi(1-2\delta)^2}+\frac{o(n)}{n}.\]

\subsection*{Plan of the paper.} Section \ref{sec-pre} provides some necessary background from \cite{BM04} and basic properties on two-dimensional hyperbolic geometry. In Section \ref{sec-bounds}, based on \cite{MR1677565, BM04} we provide several bounds on hyperbolic lengths and areas, and also give a special division for certain disks from a decomposition of $S^C(\Gamma,\mathcal{O})$ which is important in the proof of Theorem \ref{mt-i} (see Lemma \ref{l-decom}). We prove our main result Theorem \ref{mt-i} in Section \ref{sec-proof}. 

\subsection*{Acknowledgements.}
The authors would like to thank the anonymous referees for their careful reading and valuable comments, and especially would like to thank one referee for sharing his/her idea on how to obtain the current upper bound $\frac{3}{2\pi}$ in Theorem \ref{mt-i}, improving our former one $\frac{2}{3}$. They significantly  improve the quality of this paper. We also would like to thank Yuhao Xue for helpful discussions, and thank Bram Petri and Zeev Rudnick for their interests and comments on this project. The second named author is partially supported by the NSFC grant No. $12171263$.

\tableofcontents

\section{Preliminary}\label{sec-pre}
The Belyi surfaces are compact Riemann surfaces which can be defined over the algebra number field $\overline{\mathbb{Q}}$. From \cite[Theorem 1]{Be79} we know that a compact Riemann surface $S$ is a Belyi surface if and only if there exists a covering $f:S\to\mathbb{CP}^1$ unramified outside $\{0,1,\infty\}$. In this section, we will mainly review the construction of Belyi surfaces as in \cite{BM04} by Brooks and Makover. For related notations and properties, the readers may also refer to \cite[section 4]{BM04} for more details.
\subsection{Two dimensional hyperbolic geometry}
Let $\mathbb{H}=\{x+y\textbf{i};\ y>0\}$ be the upper half-plane endowed with the standard hyperbolic metric $ds^2=\frac{dx^2+dy^2}{y^2}.$
\begin{figure}[ht]
\centering
\includegraphics[width=0.40\textwidth]{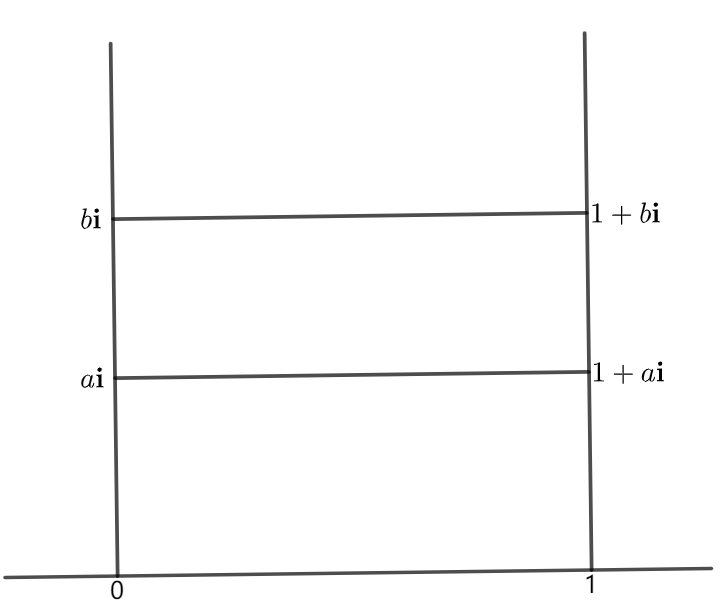}
\caption{Stripe domain}
\label{fig:00}
\end{figure}

For any $0<a<b$, set
\begin{align*}
u_a=\{t+a\textbf{i};\ 0\leq t\leq 1\} \ \text{and} \ v_{a,b}=\{t\textbf{i};\ a\leq t\leq b\}.
\end{align*}
Then their hyperbolic lengths satisfy 
\begin{align}\label{e-1ength}
\ell(u_a)=\frac{1}{a}\text{ and }\ell(v_{a,b})=\log\frac{b}{a}.
\end{align}
We also set
\begin{align*}
St(a)=\{x+y\textbf{i};\ 0\leq x\leq 1,a< y\}.
\end{align*}
Then its hyperbolic area satisfies 
\begin{align}\label{c-2}
\text{Area}(St(a))=\frac{1}{a}.
\end{align}

\subsection{Construction of Belyi surfaces}\label{s-1}
Now we recall the construction of Belyi surfaces in \cite{BM04}, i.e., the surface $S^C(\Gamma,\mathcal{O})$ associated with an oriented graph $(\Gamma,\mathcal{O})$. Let $\Gamma$ be a finite $3-$regular graph and $V(\Gamma)$ be the set of all its vertices. An orientation $\mathcal{O}$ on $\Gamma$ is an assignment, for each vertex $v\in V(\Gamma)$, of a cyclic ordering for the three edges emanating from $v$. Set
\begin{align}\label{def-Fn}
\mathcal{F}^\star_n=\left\{(\Gamma,\mathcal{O});\begin{matrix} &\Gamma\text{ is a $3$-regular graph with $2n$ vertices}\\ &\text{and }\mathcal{O}\text{ is an orientation on }\Gamma\end{matrix}\right\}.
\end{align}
Up to a M\"obius transformation, one may assume that an ideal triangle $T$ always has vertices $0,1$ and $\infty$ (see Figure \ref{fig:01}). The solid segments are geodesics joining the point $\frac{1+\sqrt3 \textbf{i}}{2}$ and the points in $\left\{\textbf{i},\ 1+\textbf{i}, \ \frac{1+\textbf{i}}{2}\right\}$. The dotted segments are horocycles joining pairs of points in $\left\{\textbf{i},\ 1+\textbf{i}, \ \frac{1+\textbf{i}}{2}\right\}$. The ideal triangle $T$ has a natural clockwise orientation $$\left(\textbf{i},\ 1+\textbf{i}, \ \frac{1+\textbf{i}}{2}\right).$$
In this article, similar as in \cite{BM04}, the points $\left\{\textbf{i},\ 1+\textbf{i}, \ \frac{1+\textbf{i}}{2}\right\}$ are called the \emph{mid-points} of the three sides of $T$, even each side has infinite length. And a dotted segment (see Figure \ref{fig:01}) joining two mid-points of an ideal triangle is always called  a \underline{short horocycle segment}. It is clear that each short horocycle segment has hyperbolic length equal to $1$.

\begin{figure}[ht]
\centering
\includegraphics[width=0.48\textwidth]{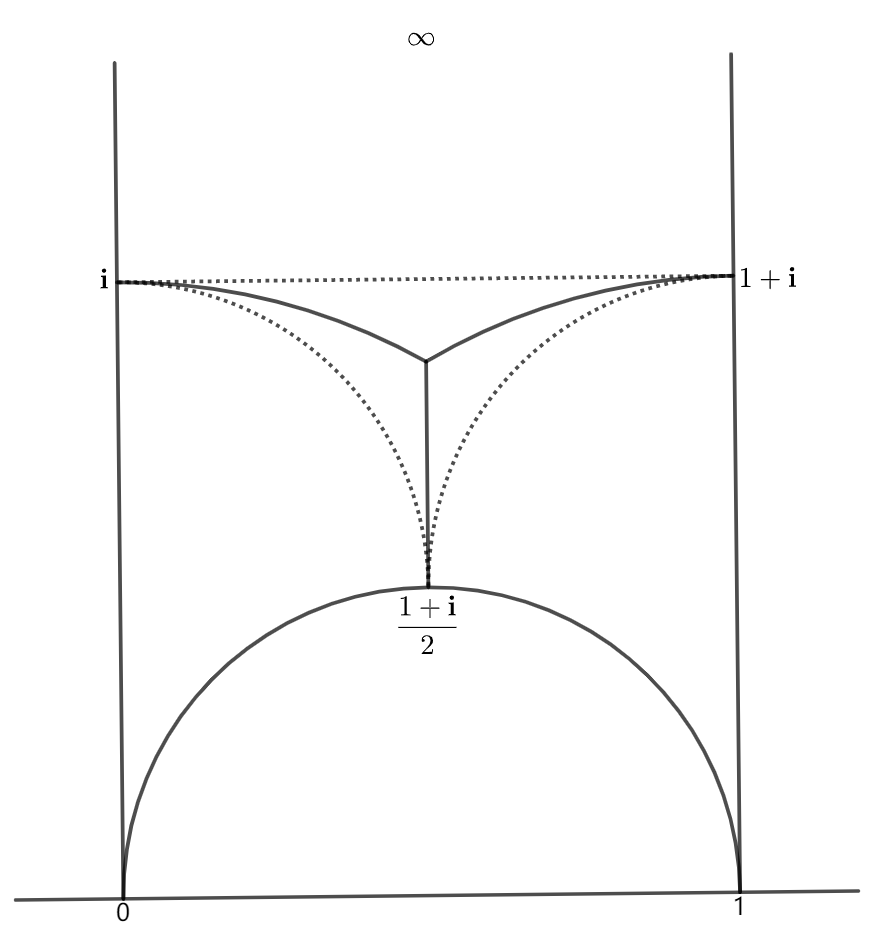}
\caption{Mid-points and short horocycle segments}
\label{fig:01}
\end{figure}

   Given an element $(\Gamma,\mathcal{O})\in\mathcal{F}^\star_n$, we replace each vertex $v\in V(\Gamma)$ by a copy of $T$, such that the natural clockwise orientation of $T$ coincides with the orientation of $\Gamma$ at the vertex $v$. If two vertices of $\Gamma$ are joined by an edge, we glue the two copies of $T$ along the corresponding sides subject to the following conditions:
\begin{enumerate}[(i)]
\item the mid-points of two sides are glued together;
\item the gluing preserves the orientations of two copies of $T$.
\end{enumerate}
As in \cite{BM04}, the surface $S^O(\Gamma,\mathcal{O})$ is uniquely determined by the two conditions above, and it is a complete hyperbolic surface with area equal to $2\pi n$.
\begin{Def*}
The compact Riemann surface $S^C(\Gamma,\mathcal{O})$ is defined as the conformal compactification of $S^O(\Gamma,\mathcal{O})$ by filling in all the punctures.
\end{Def*}
\noindent It is known that a Riemann surface $S$ is a Belyi surface if and only if $S$ can be represented as $$S=S^C(\Gamma,\mathcal{O})$$ for some $(\Gamma,\mathcal{O})\in\mathcal{F}_n^\star$ (e.g., see \cite[Lemma 2.1]{MR2271484}).

For any points $p,q,r\in\mathbb{R}\cup\{\infty\}$, denote by $\mathfrak{L}(p,q)$ the hyperbolic geodesic line joining $p$ and $q$, and denote by $\Delta(p,q,r)$ the ideal triangle with vertices $p,q$ and $r$. Let $L$ and $R$ be the matrices as follows:
$$R=\begin{pmatrix}1 & 1\\ 0 & 1\end{pmatrix}\text{ and } L=\begin{pmatrix}1 & 0\\ 1 & 1\end{pmatrix}.$$
They represent the following two automorphisms of $\mathbb{H}$:
\begin{align*}
R(z)=z+1\text{ and } L(z)=\frac{z}{z+1}.
\end{align*}
Then $L$ and $R$ generate a subgroup $G$ of $\mathrm{PSL(2,\mathbb{Z})}$. Set
$$G(T)=\{g(T);\ g\in G\}.$$
Then the set $G(T)$ consists of ideal triangles forming a partition of the upper half-plane $\mathbb{H}$ as in Figure \ref{fig:11-0}.

\begin{figure}[ht]
\centering
\includegraphics[width=0.65\textwidth]{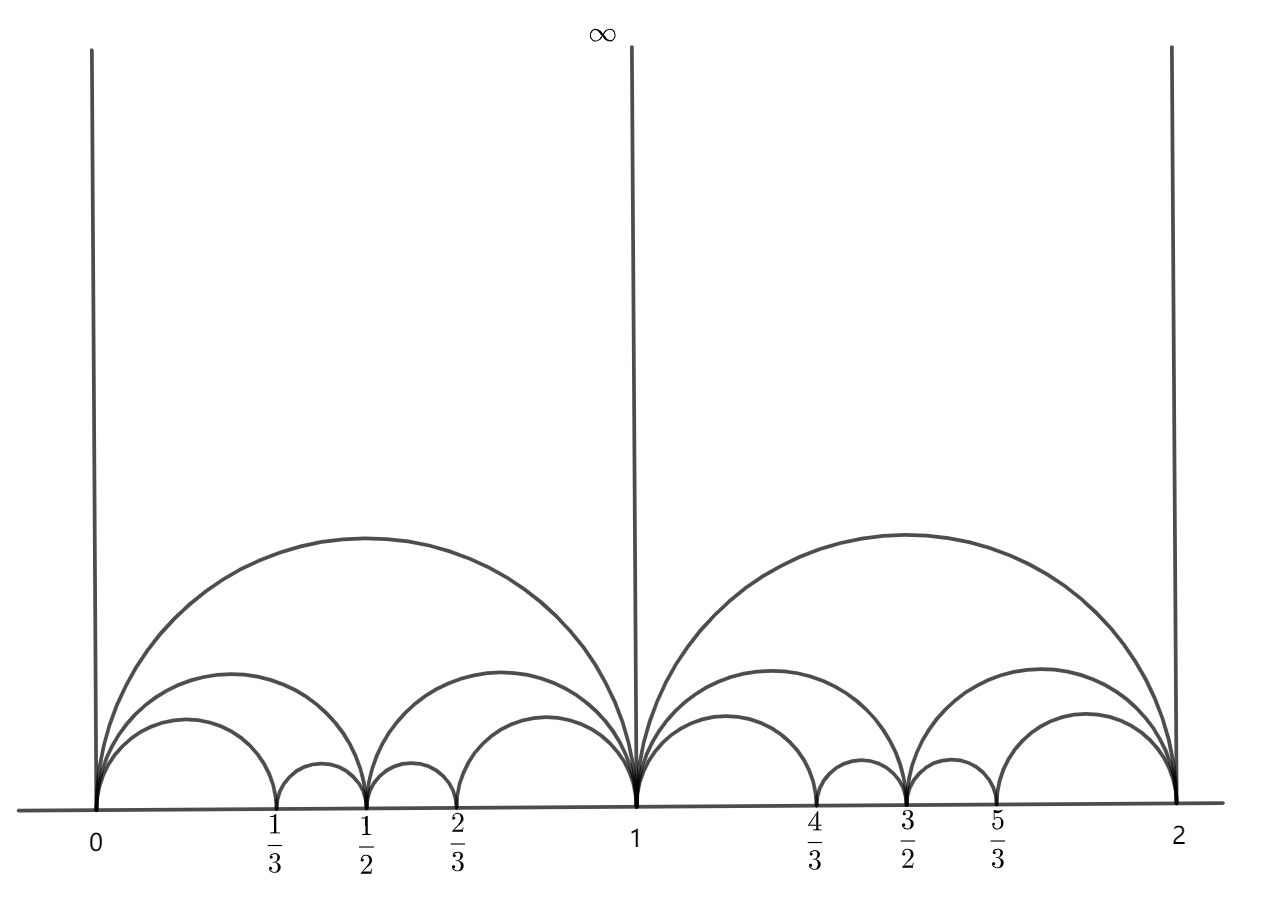}
\caption{$G(T)$ is a tiling of $\H$}
\label{fig:11-0}
\end{figure}

There exists a fundamental domain of $S^O(\Gamma,\mathcal{O})$ in $\mathbb{H}$ such that it is a finite union of ideal triangles in $G(T)$. For example, let $v_1$ and $v_2$ be two vertices of a  $3-$regular graph $\Gamma$ as shown in Figure \ref{fig:02}, and the orientations $\mathcal{O}_1 $ at $v_1$ and $v_2$ are $(1,2,3)$ and $(1,3,2)$ respectively. From the construction of $S^O(\Gamma,\mathcal{O}_1)$, it is not hard to see that the union $$\Delta(0,1,\infty)\cup\Delta(1,2,\infty)$$
is a fundamental domain of $S^O(\Gamma,\mathcal{O}_1)$ where the boundary sides $\mathfrak{L}(0,1)$ and $\mathfrak{L}(1,2)$ are identified, and the boundary sides $\mathfrak{L}(0,\infty)$ and $\mathfrak{L}(2,\infty)$ are identified. Then $S^O(\Gamma,\mathcal{O}_1)$ is a three punctured sphere (see Figure \ref{fig:02}). Using the same $3-$regular graph $\Gamma$, consider the orientation $\mathcal{O}_2$ such that the orientations at the two vertices are both $(1, 2, 3)$ as shown in Figure \ref{fig:02}, then $S^O(\Gamma,\mathcal{O}_2)$ is a torus with one puncture.

  \begin{figure}[ht]
\centering
\includegraphics[width=.9\textwidth]{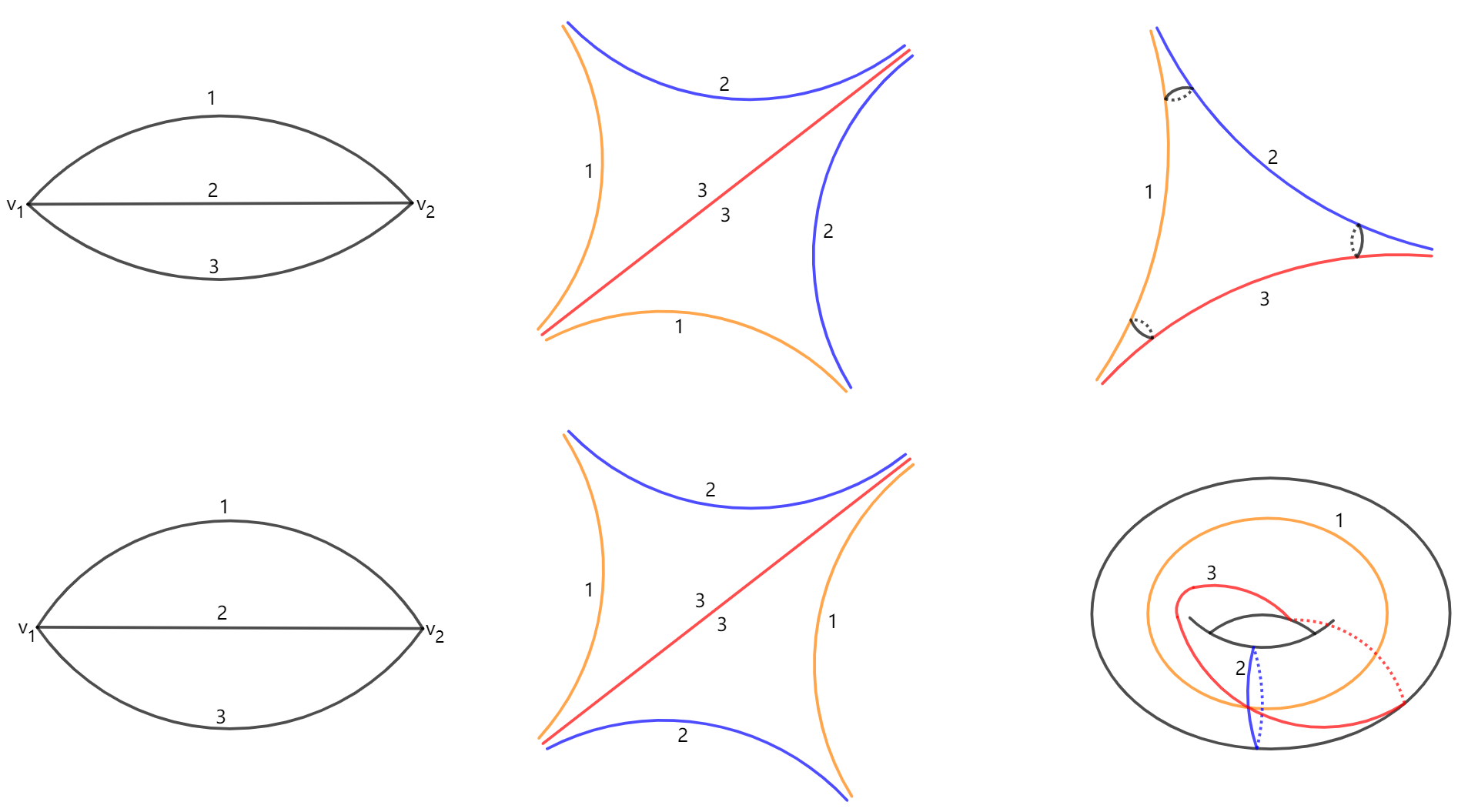}
\caption{One graph with two orientations}
\label{fig:02}
\end{figure}

\subsection{Horocycles around punctures}
Recall that \cite[Definition 4.1]{BM04} says that a \emph{left-hand-turn path} on $(\Gamma,\mathcal{O})$ is a closed path on $\Gamma$ such that at each vertex, the path turns left in the orientation $\mathcal{O}$. Every left-hand-turn path of $(\Gamma,\mathcal{O})$ corresponds to a horocycle loop in $S^O(\Gamma,\mathcal{O})$ which encloses a puncture point and consists of certain short horocycle segments as described in subsection \ref{s-1}. One may denote such loops by \underline{canonical horocycle loops}. For any two canonical horocycle loops, they are either disjoint or tangent to each other. We remark here that a canonical horocycle loop may also be tangent to itself. In the remaining, a tangency point always means either a tangency point between two canonical horocycle loops or a self-tangency point of one canonical horocycle loop.

\begin{rem*}
Since each ideal triangle contains three short horocycle segments and $S^O(\Gamma,\mathcal{O})$ consists of $2n$ ideal triangles, it follows that there are $6n$ short horocycle segments for each hyperbolic surface $S^O(\Gamma,\mathcal{O})$.
\end{rem*}

   For a random element $(\Gamma,\mathcal{O})\in\mathcal{F}^{\star}_n$, denote by $\mathrm{LHT(\Gamma,\mathcal{O})}$ the number of left-hand-turn paths in $(\Gamma,\mathcal{O})$ and by $\mathbb{E}_n[\mathrm{LHT(\Gamma,\mathcal{O})}]$ the expected value of $\mathrm{LHT(\Gamma,\mathcal{O})}$ over $\mathcal{F}^\star_n$. Then $\mathrm{LHT(\Gamma,\mathcal{O})}$ is also the number of punctures of $S^O(\Gamma,\mathcal{O})$, and the genus of $S^O(\Gamma,\mathcal{O})$ is equal to $1+\frac{n-\mathrm{LHT(\Gamma,\mathcal{O})}}{2}$. We enclose this section by the following estimate which is a direct consequence of \cite[Theorem 2.3]{BM04}.
 \begin{proposition}\label{p-2}
  There exist two universal constants $C_1,C_2>0$ independent of $n$ such that
     $$C_1+\log n\leq \mathbb{E}_n(\mathrm{LHT(\Gamma,\mathcal{O})})\leq C_2+\frac{3}{2}\log n.$$
 \end{proposition}
\noindent Gamburd showed in \cite[Corollary 5.1]{MR2271484} that $\mathbb{E}_n(\mathrm{LHT(\Gamma,\mathcal{O})}) \sim \log(3n)$ as $n\to \infty$. In this paper Proposition \ref{p-2} of Brooks-Makover is enough for us where we only need the growth rate $\log n$. For other geometric quantities of random surfaces in this model, one may also see Petri \cite{Petri17} for the behavior of the systole function; and see Budzinski-Curien-Petri \cite{BCP19} for the behavior of the diameter.

\section{Bounds on lengths and areas}\label{sec-bounds}
We start this section with the following assumption.
\begin{Assum*}
Throughout this paper, we always assume $\mathrm{genus}(S^C(\Gamma,\mathcal{O}))\geq 2$. So one may let $ds_{S^O}^2$ and $ds_{S^C}^2$ be the unique hyperbolic metrics on $S^O(\Gamma,\mathcal{O})$ and $S^C(\Gamma,\mathcal{O})$ associated to their complex structures respectively.
\end{Assum*}

For any curve $\gamma\subset S^O(\Gamma,\mathcal{O})\subset S^C(\Gamma,\mathcal{O})$, we set
\begin{align*}
&\ell_O(\gamma)=\text{the length of }\gamma\text{ under the metric }ds_{S^O}^2;\\
&\ell_C(\gamma)=\text{the length of }\gamma\text{ under the metric }ds_{S^C}^2.
\end{align*}
For any subdomain $\Omega\subset S^O(\Gamma,\mathcal{O})\subset S^C(\Gamma,\mathcal{O})$, we also set
\begin{align*}
&\mathrm{Area}_{O}(\Omega)=\text{the area of }\Omega\text{ under the metric }ds_{S^O}^2;\\
&\mathrm{Area}_{C}(\Omega)=\text{the area of }\Omega\text{ under the metric }ds_{S^C}^2.
\end{align*}
Under such notations, by Gauss-Bonnet we have
\be
\mathrm{Area}_{O}(S^O(\Gamma,\mathcal{O}))=2\pi n \text{ \ and \ } \mathrm{Area}_{C}(S^C(\Gamma,\mathcal{O}))=2\pi(n-\mathrm{LHT(\Gamma,\mathcal{O})}).
\ene

\subsection{Schwarz's Lemma} The following estimate is well-known to experts. We prove it for completeness here.
\begin{lemma}\label{l-sch}
Assume $\gamma \subset S^O(\Gamma,\mathcal{O})$ is a smooth curve and $\Omega\subset S^O(\Gamma,\mathcal{O})$ is a subdomain, then we have
$$\ell_C(\gamma)\leq \ell_O(\gamma)\text{ and }\textnormal{Area}_C(\Omega)\leq\textnormal{Area}_O(\Omega).$$
In particular, for any short horocycle segment $\gamma \subset S^O(\Gamma,\mathcal{O})$, we have
$$\ell_C(\gamma)\leq \ell_O(\gamma)=1.$$
\end{lemma}
\begin{proof}
Let $\iota$ be the natural holomorphic embedding
$$\iota:S^O(\Gamma,\mathcal{O})\to S^C(\Gamma,\mathcal{O}).$$
It suffices to prove that for any point $p\in S^O(\Gamma,\mathcal{O})\subset S^C(\Gamma,\mathcal{O})$ and any tangent vector $\nu\in T_p S^O(\Gamma,\mathcal{O})$,
 $$|\nu|_{ds_{S^O}^2}\geq |\iota_\star(\nu)|_{ds_{S^C}^2}.$$
Consider the following commutative diagram
$$\xymatrix{
\mathbb{H} \ar[d]^{\pi_1} \ar[r]_f & \mathbb{H}\ar[d]_{\pi_2}\\
S^O(\Gamma,\mathcal{O}) \ar[r]_{\iota} & S^C(\Gamma,\mathcal{O})
}
$$
 where $\pi_1$ and $\pi_2$ are covering maps. Since $\mathbb{H}$ is simply connected, there exists a holomorphic lift $f:\mathbb{H}\to\mathbb{H}$ of $\iota$, i.e.,
 \begin{align}\label{e--2}
 \pi_2\circ f=\iota\circ \pi_1.
 \end{align}
Since $f$ is holomorphic, it follows by the standard Schwarz's Lemma that $f$ is $1$-Lipschitz. Then the conclusion follows because both the covering maps $\pi_1$ and $\pi_2$ are local isometries.
 \end{proof}

\subsection{Large cusps condition}
Given $(\Gamma,\mathcal{O})\in\mathcal{F}_n^\star$ that is defined in \eqref{def-Fn}, we let $\{L_i\}_{i\in\mathcal{I}(\Gamma,\mathcal{O})}$ denote the set of all canonical horocycle loops of $S^O(\Gamma,\mathcal{O})$. Write the index set $\mathcal{I}(\Gamma,\mathcal{O})$ as $\mathcal{I}$ for simplicity. For each $i\in\mathcal{I}$, the canonical horocycle loop $L_i$ bounds a punctured disk $N_i \subset S^O(\Gamma,\mathcal{O})$ with a puncture $p_i$. Moreover $D_i=N_i\cup \{p_i\}$ is an open topological disk in $S^C(\Gamma,\mathcal{O})$.

Now we recall the so-called large cusps condition defined in \cite{MR1677565, BM04}. First for simplicity, we write $S^O(\Gamma,\mathcal{O})$ and $S^C(\Gamma,\mathcal{O})$ as $S^O$ and $S^C$ respectively. For any $i\in\mathcal{I}$, up to a conjugation one may lift the puncture $p_i$ to
$\infty$ in the boundary $\partial \H$ of the upper half plane $\H$. Let $\pi_i:\H\to S^O$ be the covering map with $\pi_i(\infty)=p_i$ and assume that the Mobius transformation through doing one turn around $p_i$ corresponds to the parabolic isometry $ z\mapsto z+1$. For any $l>0$, denote
 $$C_i(l)=\pi_i\left(\left\{z\in\mathbb{H};\ \textnormal{Im}\ z\geq \frac{1}{l}\right\}\right)$$
 and
 $$h_i(l)=\pi_i\left(\left\{z\in\mathbb{H};\ \textnormal{Im}\ z= \frac{1}{l}\right\}\right).$$
 Then for suitable $l>0$, the set $C_i(l)$ is a cusp around $p_i$ whose boundary curve $h_i(l)$ is the projection of a horocycle. We also denote by $C^l$ the standard cusp
 $$C^l=\left\{z\in\mathbb{H};\ \textnormal{Im }z\geq\frac{1}{l}\right\}/(z\sim z+1)$$
 endowed with the hyperbolic metric.
 Now we recall \cite[Definition 2.1]{MR1677565} or \cite[Definition 3.1]{BM04} saying that.
\begin{Def*}
Given any $l>0$, a hyperbolic surface $S^O$ has \emph{cusps of length $\geq l$} if
\begin{enumerate}
\item $C_i(l)$ is isometric to $C^l$ for any $i\in\mathcal{I}$;
\item $C_{i_1}(l)\cap C_{i_2}(l)=\emptyset$ for any $i_1 \neq i_2\in\mathcal{I}$.
\end{enumerate}

\end{Def*}

Brooks and Makover in \cite{BM04} proved that for any $l>0$, as $n\to \infty$ a generic hyperbolic surface $S^O(\Gamma,\mathcal{O})$ has cusps of length $\geq l$. More precisely
\begin{theorem}$($\cite[Theorem 2.1]{BM04}$)$\label{t-2}
Let $(\Gamma,\mathcal{O})$ be a random element of $\mathcal{F}^\star_n$. Then for any $l>0$,
$$\lim \limits_{n\to \infty}\textnormal{Prob}_n\{S^O(\Gamma,\mathcal{O})\text{ has cusps of length}\geq l\}= 1.$$
\end{theorem}

  Let $S^C$ be the conformal compactification of $S^O$ endowed with the hyperbolic metric $ds_{S^C}^2$, i.e., $S^C=S^O\bigcup\{p_i\}_{i\in\mathcal{I}}$. For any point $p\in S^C$ and $r>0$, denote
$$\mathbb{B}(p,r)=\text{the closed geodesic ball centered at $p$ of radius $r$ under $ds_{S^C}^2$}.$$ According to the proof of \cite[Theorem 2.1]{MR1677565}, the following theorem holds.
\begin{theorem}$($\cite[Theorem 2.1]{MR1677565}$)$\label{t-5}
For every $\epsilon>0$, there exist two positive constants $l(\epsilon)$ and $r(\epsilon)$ which only depend on $\epsilon$ such that if $S^O$ has cusps of length $\geq l(\epsilon)$, then we have
\begin{enumerate}
\item Outside of $\bigcup\limits_{i\in\mathcal{I}} C_i(l(\epsilon))$ and $\bigcup\limits_{i\in\mathcal{I}}\mathbb{B}(p_i,r(\epsilon))$ we have $$\frac{1}{1+\epsilon}ds^2_{S^O}\leq ds^2_{S^C}\leq (1+\epsilon)ds^2_{S^O}.$$

\item For any $i\in\mathcal{I}$, $$\mathbb{B}(p_i,(1+\epsilon)^{-\frac{3}{2}}r(\epsilon))\subset C_i(l(\epsilon))\cup\{p_i\}\subset\mathbb{B}(p_i,(1+\epsilon)^{\frac{3}{2}}r(\epsilon)).$$
\end{enumerate}
\end{theorem}
\begin{rem*}
Part $(2)$ above is contained in the proof of \cite[Theorem 2.1]{MR1677565} (see \cite[Page 164]{MR1677565}). Moreover as in \cite{MR1677565}, the chosen constants $l(\epsilon)\to \infty$ and $r(\epsilon)\to \infty$ as $\epsilon\to 0$. Actually the two constants $l(\epsilon)$ and $r(\epsilon)$ satisfy that $$l(\epsilon)=\frac{2\pi}{\ln \left(\frac{e^{r(\epsilon)}+1}{e^{r(\epsilon)}-1} \right)}.$$
\end{rem*}

\noindent Now we take $\epsilon=1$ in Theorem \ref{t-5} and fix the constant $r(1)>0$. From the remark above, we take $0<a<1$ such that
\begin{align}\label{e-take}
r(a)>8r(1)
\end{align}
which will be applied in the subsequent subsection. Consider the following subset of $\mathcal{F}^\star_n$ which has large proportion to the whole set. More precisely, for any $c>0$, we define
 \begin{align}\label{def-Fnc}
\mathcal{F}^\star_n(c)\overset{\text{def}}{=}\left\{(\Gamma,\mathcal{O})\in\mathcal{F}^\star_n;\begin{matrix}
&S^O(\Gamma,\mathcal{O})\text{ has cusps of length}\geq l(a)\\ &\text{ and }\mathrm{LHT(\Gamma,\mathcal{O})}\leq c\log n\end{matrix}\right\}.
\end{align}
By Proposition \ref{p-2} and Markov's inequality we have
\begin{eqnarray*}
 \textnormal{Prob}_n\{(\Gamma,\mathcal{O})\in\mathcal{F}^\star_n; \ \mathrm{LHT(\Gamma,\mathcal{O})}>c\log n\}<\frac{C_2+\frac{3}{2}\log n}{c\log n}
\end{eqnarray*}
which together with Theorem \ref{t-2} implies that for $n$ large enough,
 \begin{align}\label{e-0}
  \textnormal{Prob}_n\{(\Gamma,\mathcal{O})\in\mathcal{F}^\star_n(c)\}\geq 1-\frac{2}{c}.
  \end{align}
Since the genus of $S^C(\Gamma,\mathcal{O})$ is equal to $1+\frac{n-\mathrm{LHT(\Gamma,\mathcal{O})}}{2}$,
we have that for any $(\Gamma,\mathcal{O})\in\mathcal{F}^\star_n(c)$,
  \begin{align}\label{e-coarea}
  \textnormal{Area}_C\left(S^C(\Gamma,\mathcal{O})\right)=2\pi\left(n-\textnormal{LHT}(\Gamma,\mathcal{O})\right) \geq 2\pi(n-c\cdot\log n).
  \end{align}
\subsection{Divisions of disks}
Let $(\Gamma,\mathcal{O})\in\mathcal{F}^\star_n(c)$ that is defined in \eqref{def-Fnc}. For each $i\in\mathcal{I}$, we assume that the canonical horocycle loop $L_i$ around $p_i$ consists of $d_i$ short horocycle segments. So we have \be\label{eq-di}
\ell_O(L_i)=d_i  \text{\ \ and } \ \textnormal{Area}_O(N_i)=d_i \ene
where $N_i \subset S^O(\Gamma,\mathcal{O})$ is the punctured disk enclosed by the canonical horocycle loop $L_i$ around $p_i$ in $S^O(\Gamma,\mathcal{O})$. Recall that each $S^O(\Gamma,\mathcal{O})$ consists of $2n$ ideal triangles each of which contains three short horocycle segments. So we have \begin{align}\label{e-sum}\sum\limits_{i\in\mathcal{I}}d_i=6n.\end{align}
Divide $\mathcal{I}$ into the following two subsets
\begin{align}
\mathcal{I}_1\overset{\text{def}}{=}\left\{i\in\mathcal{I};\ d_i>\frac{n}{(\log n)^2}\right\}\text{ and \  }
\mathcal{I}_2\overset{\text{def}}{=}\left\{i\in\mathcal{I};\ d_i\leq\frac{n}{(\log n)^2}\right\}.
\end{align}
If $i\in\mathcal{I}_1$, then for $n$ large enough we have
   \begin{align}\label{e-9}d_i>\frac{n}{(\log n)^2}>l(a).\end{align}
Recall that for every $i\in\mathcal{I}$, $D_i=N_i\cup \{p_i\}$ is a topological disk in $S^C(\Gamma,\mathcal{O})$. For any $i\in\mathcal{I}_1$, as in Figure \ref{fig:07}, there are four closed loops around the puncture point $p_i$ as follows:
\begin{enumerate}[(a)]
 \item $L_i$ is the canonical horocycle loop under the hyperbolic metric $ds_{S^O}^2$;
 \item $h_i(l(a))$ is the horocycle loop with length  $l(a)$ under the hyperbolic metric $ds_{S^O}^2$;
 \item $h_i(l(1))$ is the horocycle loop with length $l(1)$ under the hyperbolic metric $ds_{S^O}^2$;
 \item  the loop $\partial \mathbb{B}(p_i,2\sqrt 2r(1))$ is the boundary of the geodesic ball $\mathbb{B}(p_i,2\sqrt 2r(1))$ under the hyperbolic metric $ds_{S^C}^2$.
\end{enumerate}
\begin{figure}[ht]
\centering
\includegraphics[width=0.5\textwidth]{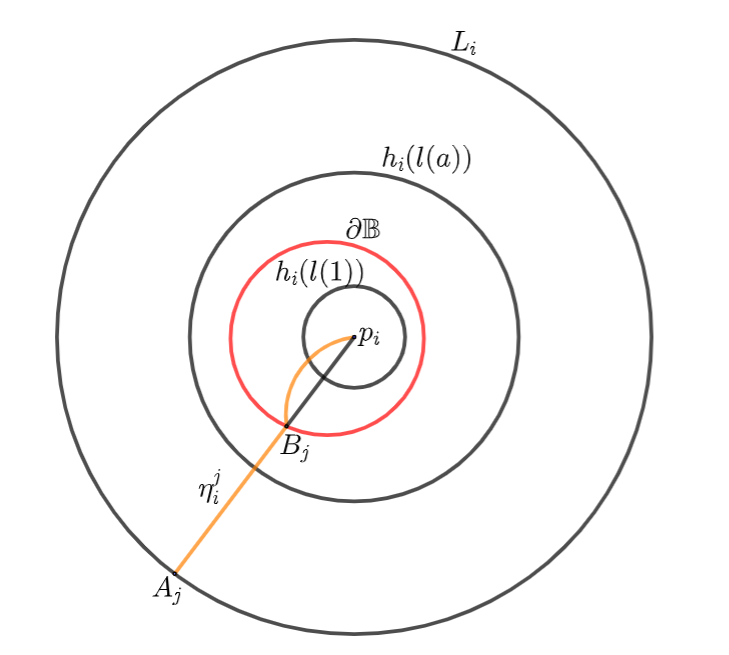}
\caption{Four loops around one puncture}
\label{fig:07}
\end{figure}
Then combining with \eqref{e-take}, \eqref{e-9} and Part $(2)$ of Theorem \ref{t-5}  we have
\be \label{multi-incl}
C_i(l(1))\cup\{p_i\}\subset\mathbb{B}(p_i,2\sqrt{2}r(1))
\subset \mathbb{B}\left(p_i,\frac{r(a)}{(1+a)^{3/2}}\right)\subset C_i(l(a))\cup \{p_i\}\subset D_i
\ene
which in particular tells that the four loops $L_i, h_i(l(a)), h_i(l(1))$ and $\partial \mathbb{B}(p_i,2\sqrt 2r(1))$ are pairwisely disjoint.

Assume that the set of all tangency points on $L_i$ is $$\{A_j\}_{1\leq j \leq d_i}$$ arranged in the anti-clockwise direction. For each $1\leq j\leq d_i$, we let $p_iA_j$ be the geodesic ray joining $p_i$ and $A_j$ under the hyperbolic metric $ds^2_{S^O}$. We remark here that there may exist several intersection points between $p_iA_j$ and the loop $\partial\mathbb{B}(p_i,2\sqrt 2r(1))$. Now define $B_{j}$ to be the last intersection point between them on the direction from $p_i$ to $A_j$, i.e., $B_j$ is the unique point on the ray $p_iA_j$ such that $$B_jA_j\cap \partial\mathbb{B}(p_i,2\sqrt 2r(1))=B_j$$ where $B_{j}A_j$ is the subsegment of $p_iA_j$ joining $B_{j}$ and $A_j$. By \eqref{multi-incl} we know that $\mathbb{B}(p_i,2\sqrt2 r(1))\subset C_i(l(a))\cup \{p_i\}$ where $C_i(l(a))\cup \{p_i\}$ is topologically a disk. Then for any point $Q\in \mathbb{B}(p_i,2\sqrt2 r(1))$, there exists a unique shortest geodesic $\overline{p_iQ}$ joining $p_i$ and $Q$
under the hyperbolic metric $ds_{S^C}^2$. In particular, we let $\overline{p_iB_{j}}$ be the unique shortest geodesic joining $p_i$ and $B_j$. Now consider the concatenation (see Figure \ref{fig:07} for an illustration)
\be\label{eta-i}
\eta_i^j=\overline{p_iB_{j}}\cup B_{j}A_j.
\ene
By construction, it is not hard to see that
\ben
\item for $1\leq j\neq k\leq d_i$, $\eta_i^j\cap \eta_i^k=\{p_i\}$;

\item the curve $\eta_i^j$ joins $p_i$ and $A_j$ with $\eta_i^j \subset D_i$.
\een
Moreover, the length $\ell_C(\eta_i^j)$ can be effectively bounded from above. More precisely, it follows by \eqref{e-1ength}, \eqref{e-sum} and Lemma \ref{l-sch} that for each $1\leq j \leq d_i$ and $n$ large enough,
\begin{align}\label{e-curve}
\ell_C(\eta_i^j)&=\ell_C\left(\overline{p_iB_{j}}\right)+\ell_C\left(B_{j}A_j\right)\\
&\leq \ell_C\left(\overline{p_iB_{j}}\right)+\ell_O\left(B_{j}A_j\right)\nonumber\\
&< 2\sqrt2 r(1)+\log \frac{d_i}{l(1)}\nonumber\\
&\leq 2\log n. \nonumber
\end{align}

  Assume that $\alpha,\beta \subset D_i$ are two simple curves joining $p_i$ and certain points on $L_i$ such that $\alpha\cap\beta=\{p_i\}$ and both $\alpha$ and $\beta$ only intersect with $L_i$ at their endpoints. Denote by $D(\alpha,\beta)$ the domain enclosed by $\alpha,\ \beta$ and $L_i$ in the anti-clockwise direction from $\alpha$ to $\beta$ (see Figure \ref{ab-d} for an illustration).
\begin{figure}[ht]
\centering
\includegraphics[width=0.45\textwidth]{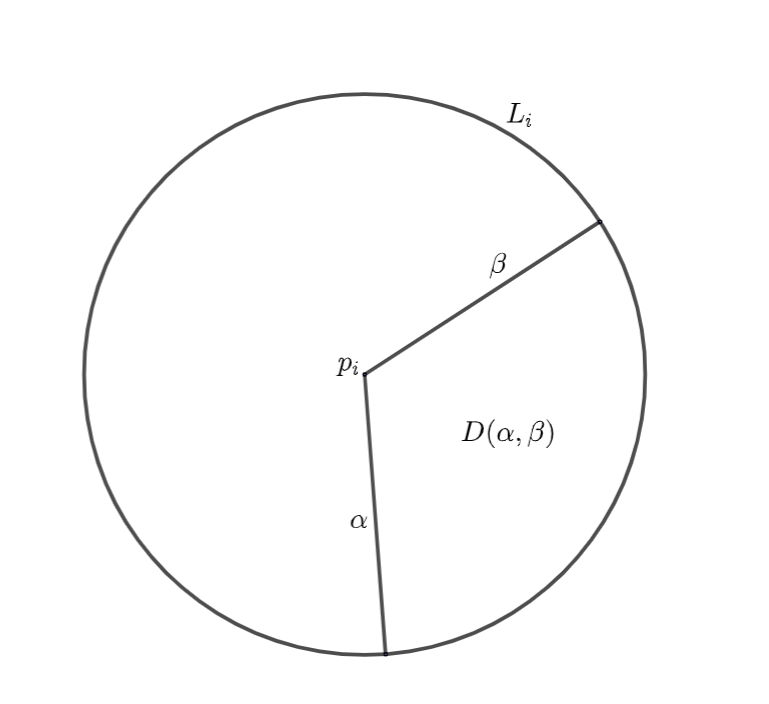}
\caption{A domain $D(\alpha, \beta)$}
\label{ab-d}
\end{figure}
 For any $1\leq j\neq k\leq d_i$, for simplicity we denote $$
  \textnormal{Area}_O^i(j,k)=\textnormal{Area}_O\left(D\left(\eta_i^j,\eta_i^k\right)\right)\text{ and }
   \textnormal{Area}_C^i(j,k)=\textnormal{Area}_C\left(D\left(\eta_i^j,\eta_i^k\right)\right).$$

Now we give a division for each $D_i$ $(i\in\mathcal{I}_1)$ and prove the following result which is important in the proof of Theorem \ref{mt-i}.
\begin{lemma}[A division of $D_i$]\label{l-decom}
 Assume $(\Gamma,\mathcal{O})\in\mathcal{F}_n^\star(c)$ that is defined in \eqref{def-Fnc} and $i\in\mathcal{I}_1$, then for $n$ large enough, there exist a positive integer $k_i$ with $k_i\leq 2(\log n)^2$ and a sequence of integers $\{m_j\}_{1\leq j \leq k_i}$ with
 $$1=m_1<m_2<...<m_{k_i}\leq d_i$$
 such that for each $1\leq j\leq k_i-1$
\be
\textnormal{Area}_C^i(m_j,m_{j+1})\leq\frac{2d_i}{(\log n)^2}\text{ and }m_{j+1}-m_j\leq\frac{3d_i}{(\log n)^2},
\ene
and
\be
\textnormal{Area}_C^i(m_{k_i},m_{1})\leq\frac{2d_i}{(\log n)^2}\text{ and }m_1+d_i-m_{k_i}\leq\frac{3d_i}{(\log n)^2}.
\ene
\end{lemma}
\begin{proof}
Recall that $D_i=N_i \cup \{p_i\}$. First it follows from \eqref{eq-di} and Lemma \ref{l-sch} that
\begin{align}\label{e-ub}\textnormal{Area}_C(D_i)=\textnormal{Area}_C(N_i)\leq \textnormal{Area}_O(N_i)=d_i.\end{align}
Notice that $S^C\setminus D_i=S^O\setminus N_i$. Similarly, it follows by Lemma \ref{l-sch} that $\textnormal{Area}_C(S^C\setminus D_i)\leq \textnormal{Area}_O(S^O\setminus N_i)$ implying
$$\textnormal{Area}_O(N_i)-\textnormal{Area}_C(D_i)\leq \textnormal{Area}_O(S^O)-\textnormal{Area}_C(S^C)\leq 2\pi c\cdot\log n,$$
where we apply our assumption $(\Gamma,\mathcal{O})\in\mathcal{F}_n^\star(c)$ in the last inequality.
Since $i\in\mathcal{I}_1$, we have that for $n$ large enough,
\begin{align}\label{e-lb}
\textnormal{Area}_C(D_i)\geq d_i-2\pi c\cdot\log n\geq \frac{n}{(\log n)^2}-2\pi c\cdot\log n.
\end{align}
Recall that for any $1\leq j\leq d_i$, $p_iA_j$ is the geodesic ray joining $p_i$ and $A_j$ under the hyperbolic metric $ds^2_{S^O}$. Then it follows from \eqref{c-2}, \eqref{e-lb} and Lemma \ref{l-sch} that for any $1\leq j\leq d_i-1$ and $n$ large enough,
\begin{align}\label{e-domain}
   \textnormal{Area}_C^i(j,j+1)&\leq\textnormal{Area}_C\left(\mathbb{B}\left(p_i,2\sqrt2 r(1)\right)\right)+\textnormal{Area}_C\left(D\left(p_iA_j,p_iA_{j+1}\right)\right)\\
&\leq 4\pi\left(\cosh \left(2\sqrt2 r(1)\right)-1\right)+\textnormal{Area}_O\left(D\left(p_iA_j,p_iA_{j+1}\right)\right) \nonumber\\
&= 4\pi\left(\cosh \left(2\sqrt2 r(1)\right)-1\right)+1\nonumber \\
&\leq \frac{\textnormal{Area}_C(D_i)}{(\log n)^2}. \nonumber
\end{align}
Similarly, for $j=d_i$ we also have
\be
 \textnormal{Area}_C^i(d_i,d_i+1)\leq \frac{\textnormal{Area}_C(D_i)}{(\log n)^2}.
\ene

\noindent Choose an integer $k_i \in [2,d_i]$ and a sequence of increasing integers $\{m_j\}_{j=1}^{k_i}$ with $m_{k_i}\leq d_i$
such that they satisfy the following three conditions: \begin{enumerate}[(a)]
 \item $m_1=1$;
 \item for each $1\leq j\leq k_i-1$,
 $$m_{j+1}=\min\limits_{m_j\leq m\leq d_i}\left\{m;\ \textnormal{Area}_C^i(m_j,m)\geq\frac{\textnormal{Area}_C(D_i)}{(\log n)^2}\right\};$$
 \item $\textnormal{Area}_C^i(m_{k_i},m_1)\leq\frac{2\textnormal{Area}_C(D_i)}{(\log n)^2}$.
\end{enumerate}
Now we prove that the $k_i$ and $\{m_j\}_{j=1}^{k_i}$ are the desired integer and sequence. First from \eqref{e-ub} and Condition $(c)$ we have
\be
\textnormal{Area}_C^i(m_{k_i},m_1)\leq\frac{2\textnormal{Area}_C(D_i)}{(\log n)^2}\leq\frac{2d_i}{(\log n)^2}.
\ene
Next it follows from \eqref{e-ub}, \eqref{e-domain} and Condition $(b)$ that for each $1\leq j\leq k_i-1$,
\begin{align}\label{e-ulb}
\frac{\textnormal{Area}_C(D_i)}{(\log n)^2}\leq \textnormal{Area}_C^i(m_j,m_{j+1})\leq\frac{2\textnormal{Area}_C(D_i)}{(\log n)^2}\leq\frac{2d_i}{(\log n)^2},
\end{align}
where the second inequality holds since from \eqref{e-domain} and Condition $(b)$ we have
\begin{align*}
\textnormal{Area}_C^i(m_j,m_{j+1})&= \textnormal{Area}_C^i(m_j,m_{j+1}-1)+\textnormal{Area}_C^i(m_{j+1}-1,m_{j+1})\\
&\leq \frac{2\textnormal{Area}_C(D_i)}{(\log n)^2}.
\end{align*}
For each $1\leq j\leq k_i-1$, we have
\begin{align}\label{e-ulb-2}
&\textnormal{Area}_O(D(p_iA_{m_j},p_iA_{m_{j+1}}))-\textnormal{Area}_C(D(p_iA_{m_j},p_iA_{m_{j+1}})) \\
&\leq\textnormal{Area}_O(S^O)-\textnormal{Area}_C(S^C)\leq 2\pi c\cdot\log n.\nonumber
\end{align}
Similar as the first inequality in \eqref{e-domain} we also have
\be \label{e-ulb-3}
\textnormal{Area}_C(D(p_iA_{m_j},p_iA_{m_{j+1}}))\leq \textnormal{Area}_C\left(\mathbb{B}\left(p_i,2\sqrt2 r(1)\right)\right)+\textnormal{Area}_C(m_j,m_{j+1}).
\ene
Recall that $i\in\mathcal{I}_1$. So $d_i>\frac{n}{(\log n)^2}$. Then it follows by \eqref{e-ulb}, \eqref{e-ulb-2} and \eqref{e-ulb-3} that for  each $1\leq j\leq k_i-1$ and $n$ large enough,
\be
m_{j+1}-m_j=\textnormal{Area}_O(D(p_iA_{m_j},p_iA_{m_{j+1}}))\leq\frac{3d_i}{(\log n)^2}.
\ene

\noindent By a similar argument, if $j=k_i$, we also have
\be
m_1+d_i-m_{k_i}\leq\frac{3d_i}{(\log n)^2}.
\ene
Now it remains to bound $k_i$. From \eqref{e-ulb} we have
$$(k_i-1)\times\frac{\textnormal{Area}_C(D_i)}{(\log n)^2}\leq\sum\limits_{j=1}^{k_i-1} \textnormal{Area}_C^i(m_j,m_{j+1})\leq \textnormal{Area}_C(D_i),$$
which implies that for $n$ large enough,
\be
k_i\leq (\log n)^2+1<2(\log n)^2.
\ene

The proof is complete.
\end{proof}

\begin{rem*}
The ratio $\frac{d_i}{(\log n)^2}$ is important when we estimate the variance in the proof of Lemma \ref{l-exp-2} in the next section. And it is not hard to see that $k_i\geq \frac{(\log n)^2}{2}$. So $k_i$ is uniformly comparable to $(\log n)^2$. In this paper we will only apply the upper bound as shown in Lemma \ref{l-exp-2}.
\end{rem*}

\section{Proof of Theorem \ref{mt-i}}\label{sec-proof}
In this section, we finish the proof of Theorem \ref{mt-i}. First we always assume $(\Gamma,\mathcal{O})\in\mathcal{F}_n^\star(c)$  that is defined in \eqref{def-Fnc}. For simplicity of notation, we denote 
$$D_{ij}=D(\eta_{i}^{m_j},\eta_i^{m_{j+1}}) \text{ for }1\leq j\leq k_i-1$$
and 
$$D_{ik_i}=D(\eta_i^{m_{k_i}},\eta_i^{m_1}).$$

From Lemma \ref{l-decom}, we have the following decomposition of $S^C=S^C(\Gamma,\mathcal{O})$:
$$S^{C}=\bigcup\limits_{i\in\mathcal{I}_1}
\left(\bigcup\limits_{j=1}^{k_i}D_{ij}\right)\bigcup\left(\bigcup
\limits_{i\in\mathcal{I}_2}D_i\right)\bigcup\{\text{$2n$ small triangles}\},$$
 where each small triangle is enclosed by three short horocycle segments (see Figure \ref{fig:11} for an illustration).
\begin{figure}[ht]
\centering
\includegraphics[width=0.5\textwidth]{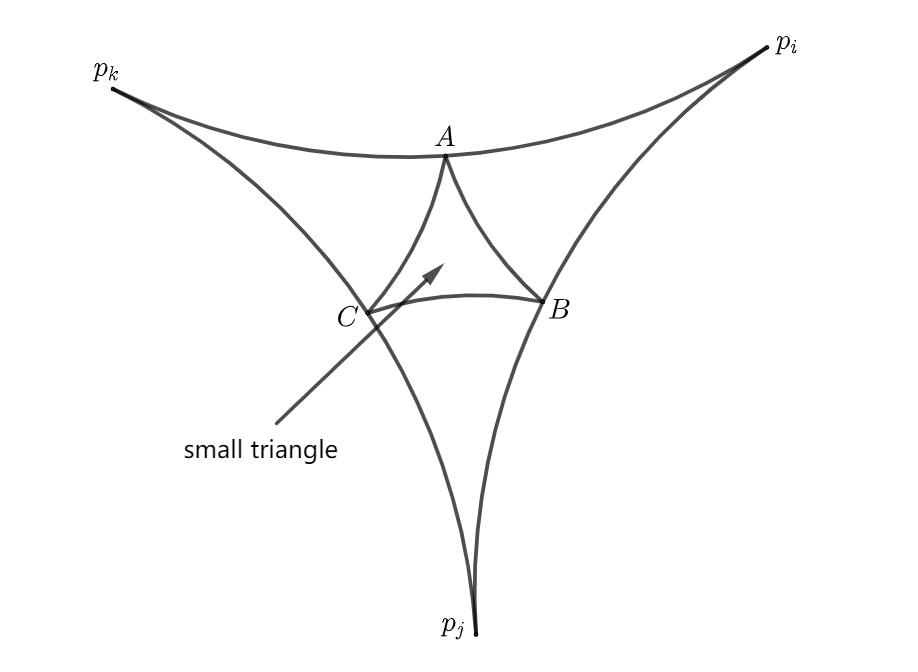}
\caption{A small triangle}
\label{fig:11}
\end{figure}

Denote
$$\mathfrak{S}^C=\{D_{ij};\ i\in\mathcal{I}_1,1\leq j\leq k_i\}\bigcup
\{D_i;\ i\in\mathcal{I}_2\}\bigcup\{\text{$2n$ small triangles}\}.$$
 For each short horocycle segment $\gamma$, it uniquely determines a domain $\Omega_{\gamma}$ amongst $$\mathfrak{S}_0^C=\{D_{ij};\ i\in\mathcal{I}_1,1\leq j\leq k_i\}\bigcup
\{D_i;\ i\in\mathcal{I}_2\}$$ such that $\gamma$ is contained in the boundary of $\Omega_{\gamma}$. Let $\Delta$ be a small triangle enclosed by three short horocycle segments $\{\gamma_i\}_{i=1}^3$. We say that $\{\Omega_{\gamma_i}\}_{i=1}^3$ are \emph{sector domains} of $\Delta$. We remark here that $\Omega_{\gamma_1}$ may be the same as $\Omega_{\gamma_2}$ for $\gamma_1\neq \gamma_2$.


Take two symbols $\mathcal{A}$ and $\mathcal{B}$, and let $J:\mathfrak{S}^C\to\{\mathcal{A},\mathcal{B}\}$ be a mapping.  Since $$J(\Omega_{\gamma_1}),  J(\Omega_{\gamma_2}),  J(\Omega_{\gamma_3}) \in \{\mathcal{A},\mathcal{B}\},$$
it follows that amongst $\{\Omega_{\gamma_i}\}_{i=1}^3$, at least two of them have the same image $I\left(\Delta\right)\in \{\mathcal{A},\mathcal{B}\}$ under the mapping $J$. We say the mapping $J$ is \emph{compatible} on $\mathfrak{S}^C$ if
$$J(\Delta)=I(\Delta)$$
for every small triangle $\Delta$.
\begin{figure}[ht]
\centering
\includegraphics[width=0.8\textwidth]{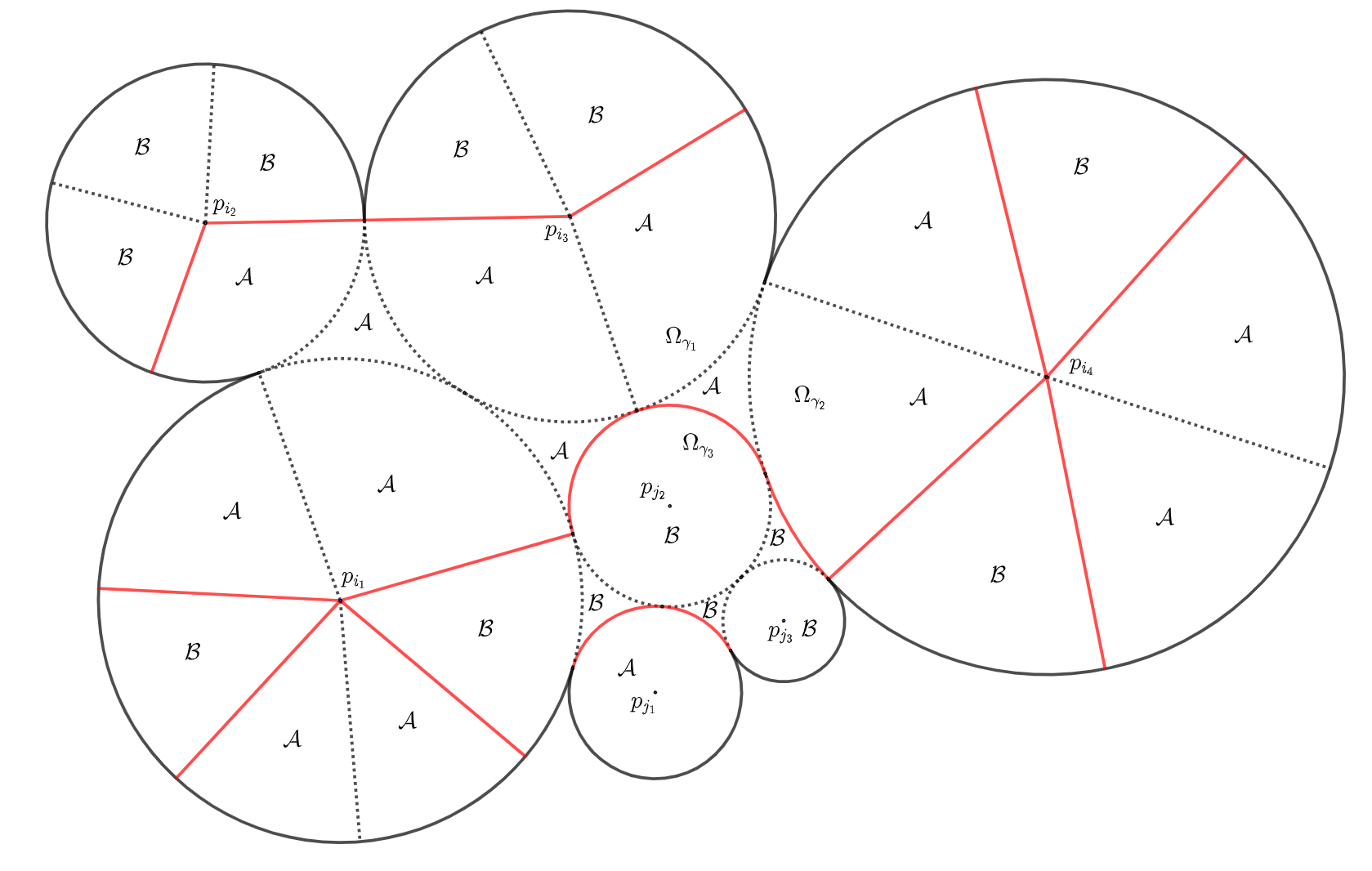}
\caption{An example for a compatible mapping $J$}
\label{fig:15}
\end{figure}
\begin{rem*}
See Figure \ref{fig:15}, for examples:
\begin{enumerate}[(a)]
\item if $J(\Omega_{\gamma_1})=J(\Omega_{\gamma_2})=\mathcal{A},J(\Omega_{\gamma_3})=\mathcal{B}$, then we have $J(\Delta)=\mathcal{A}.$ In this case both the short horocycle segments $\gamma_1$ and $\gamma_2$ are not contained in the boundary of $\mathcal{A}(J)$ and $\mathcal{B}(J)$ because they are interior points of $\mathcal{A}$.
\item If $J(\Omega_{\gamma_1})=J(\Omega_{\gamma_2})=J(\Omega_{\gamma_3})=\mathcal{A}$, then we have $J(\Delta)=\mathcal{A}$. In this case all the short horocycle segments $\{\gamma_i\}_{i=1}^3$ are not contained in the boundary of $\mathcal{A}(J)$ and $\mathcal{B}(J)$ for same reason above.
\end{enumerate}
\end{rem*}

Define
 $$X_n\overset{\text{def}}{=}\left\{\text{all compatible mappings }J:\mathfrak{S}^C\to\{\mathcal{A},\mathcal{B}\}\right\}.$$
 It is clear that $X_n$ is a finite set and there exists a uniform probability measure $\mu$ on it. For any $J\in X_n$, it will induce a division $(\mathcal{A}(J),\mathcal{B}(J))$ of $S^C$, where
\be \label{def=divi}
\mathcal{A}(J)\overset{\text{def}}{=}\bigcup\limits_{\Omega\in\mathfrak{S}^C,\ J(\Omega)=\mathcal{A}}\Omega\text{\ and \ }\mathcal{B}(J)\overset{\text{def}}{=}\bigcup\limits_{\Omega\in\mathfrak{S}^C,\ J(\Omega)=\mathcal{B}}\Omega.
\ene
It is clear that the two boundaries coincide, i.e., $\partial \mathcal{A}(J)=\partial \mathcal{B}(J)$.
\begin{remark}\label{r-bound}
For a compatible mapping $J$, it is clear that for any small triangle $\Delta$, at most one of its horocycle segment boundaries is contained in $\partial\mathcal{A}(J)=\partial\mathcal{B}(J)$.
\end{remark}
For any property $P$, denote by $\mathbb{P}\left(P\right)$ the probability that a random element $J\in (X_n,\mu)$ satisfies property $P$, i.e., 
\[\mathbb{P}\left(P\right)=\mu\left(\{J\in X_n; \ J \textit{ satisfies property $P$} \}\right).\]
Here we emphasis again that the randomness is only on $J$ and the surface $(\Gamma,\mathcal{O})\in\mathcal{F}_n^\star(c)$ is fixed.
For any random variable $Y$ on $(X_n,\mu)$, denote by $\mathbb{E}\left[Y\right]$ and $\textnormal{Var}\left(Y\right)$ the expected value and variance of $Y$ respectively. Now we consider the following three random variables:
\begin{enumerate}
\item $|\partial\mathcal{A}_n|(J)=\ell_C(\partial\mathcal{A}(J))$;
\item $|\mathcal{A}_n|(J)=\textnormal{Area}_C(\mathcal{A}(J))$;
\item $|\mathcal{MA}_n|(J)=\min\left\{\textnormal{Area}_C(\mathcal{A}(J)), \
    \textnormal{Area}_C(\mathcal{B}(J))\right\}$.
\end{enumerate}
We will bound the expected values of $|\partial\mathcal{A}_n|(J)$ and $|\mathcal{MA}_n|(J)$. The first one is
\begin{lemma}\label{l-exp-1}
For $n$ large enough, we have
$$\mathbb{E}\left[|\partial\mathcal{A}_n|\right]\leq \frac{3n}{2}+4c\cdot(\log n)^4.$$
\end{lemma}

\begin{proof}
For each element $J\in X_n$, the boundary $\partial\mathcal{A}(J)$ consists of two parts: the simple curves $\eta_{i}^{m_j}\subset D_i \ (i\in\mathcal{I}_1,\ 1\leq j\leq k_i)$ defined in \eqref{eta-i} and certain short horocycle segments. From Lemma \ref{l-decom}, we have that for any $i\in\mathcal{I}_1$ and $1\leq j\leq k_i$,
 \begin{align}\label{e-bound}
 \ell_C(\eta_{i}^{m_j})\leq 2\log n\text{ \ and \ }k_i\leq 2(\log n)^2.
 \end{align}
 Since $|\mathcal{I}_1|\leq |\mathcal{I}|\leq c\cdot\log n$, by \eqref{e-bound} we may conclude that the first part has total length $\leq 4c\cdot(\log n)^4$. Now we estimate the total length of the second part: for any small triangle $\Delta$ enclosed by three short horocycle segments $\{\gamma_i\}_{i=1}^3$, we denote
$$\mathbb{P}(\Delta)=\mathbb{P}\left(J\in X_n;\ \gamma_i\nsubseteq\partial\mathcal{A}(J)\text{ for all }i=1,2,3\right).$$
If $\Omega_{\gamma_1}=\Omega_{\gamma_2}=\Omega_{\gamma_3}$, then $\mathbb{P}(\Delta)=1$. If exactly two of $\{\Omega_{\gamma_i}\}_{i=1}^3$ are the same, then $\mathbb{P}(\Delta)=\frac{1}{2}$. If $\{\Omega_{\gamma_i}\}_{i=1}^3$ are pairwise distinct, then $\mathbb{P}(\Delta)=\frac{1}{4}$. To summarize, we have that for any small triangle $\Delta$,
\begin{align*}
\mathbb{P}(\Delta)\geq\frac{1}{4}.
\end{align*}
It follows that 
\begin{align}\label{e-pbound}
\mathbb{P}\left(J\in X_n;\ \ell_C(\partial\Delta\cap\partial\mathcal{A}(J))=0\right)\geq \frac{1}{4}.
\end{align}
Recall that for any short horocycle segment $\gamma$,
$$\ell_{C}(\gamma)\leq 1.$$
Combined with Remark \ref{r-bound}, we have that for any small triangle $\Delta$,
\begin{align}\label{n-bound}\ell_C(\partial\Delta\cap\partial\mathcal{A}(J))\leq 1.\end{align}
 Since there are $2n$ small triangles, together with \eqref{e-pbound}, \eqref{n-bound} and the property that the first part has total length $\leq 4c\cdot(\log n)^4$, we obtain
$$\mathbb{E}\left[|\partial\mathcal{A}_n|\right]\leq 2n\times\frac{3}{4}+4c\cdot(\log n)^4=\frac{3n}{2}+4c\cdot(\log n)^4$$ as desired.
\end{proof}

Set
$$L=\bigcup\limits_{i\in\mathcal{I}_1}\left(\bigcup\limits_{j=1}^{k_i}\eta_{i}^{m_j}\right)
\bigcup\{\text{$6n$ short horocycle segments}\},$$
where $\{\eta_{i}^{m_j}\}_{i\in\mathcal{I}_1,\ 1\leq j\leq k_i}$, defined in \eqref{eta-i}, are the simple curves in Lemma \ref{l-decom}. It is clear that $\nu(L)=0$ where $\nu$ is the measure induced from the hyperbolic metric on $S^C$. For any subset $\Omega\subset S^C$, the characteristic function $1_{\Omega}: S^C\to\mathbb{R}$ is defined by $$1_{\Omega}(x)=\begin{cases}1\text{\quad if }x\in\Omega;\\0\text{\quad if }x\notin\Omega.\end{cases}$$
Then the expected value satisfies
\begin{align*}
\mathbb{E}\left[|\mathcal{A}_n|\right]&=\int\limits_{X_n}\textnormal{Area}_C
\left(\mathcal{A}(J)\right)d\mu(J)=\int\limits_{X_n}\int\limits_{S^C}1_{\mathcal{A}(J)}(x)d\nu(x)d\mu(J)\\
&=\int\limits_{S^C}\int\limits_{X_n}1_{\mathcal{A}(J)}(x)d\mu(J)d\nu(x)=\int\limits_{S^C}\mathbb{P}\left(J\in X_n;\ x\in\mathcal{A}(J)\right)d\nu(x)\\
&=\int\limits_{S^C\setminus L}\mathbb{P}\left(J\in X_n;\ x\in\mathcal{A}(J)\right)d\nu(x)=\frac{1}{2}\textnormal{Area}_C\left(S^C\right),
\end{align*}
where the last equality holds since for any $x\in S^C\setminus L$, $$\mathbb{P}\left(J\in X_n;\ x\in\mathcal{A}(J)\right)=\frac{1}{2}.$$ Now we calculate $\mathbb{E}\left[|\mathcal{A}_n|^2\right]$. By definition we have
\begin{align*}
\mathbb{E}\left[|\mathcal{A}_n|^2\right]&=\int\limits_{X_n}\textnormal{Area}_C
\left(\mathcal{A}(J)\right)^2d\mu(J)\\
&=\int\limits_{X_n}\int\limits_{S^C\times S^C}1_{\mathcal{A}(J)}(x)\cdot 1_{\mathcal{A}(J)}(y)d\nu(x)d\nu(y)d\mu(J)\\
&=\int\limits_{S^C\times S^C}\mathbb{P}\left(J\in X_n;\ x,y\in\mathcal{A}(J)\right)d\nu(x)d\nu(y)\\
&=\int\limits_{(S^C\setminus L)\times (S^C\setminus L)}\mathbb{P}\left(J\in X_n;\ x,y\in\mathcal{A}(J)\right)d\nu(x)d\nu(y).
\end{align*}
Thus, the variance satisfies 
\begin{align}\label{e-var}
\textnormal{Var}\left(|\mathcal{A}_n|\right)&=\mathbb{E}\left[|\mathcal{A}_n|^2\right]-\mathbb{E}\left[|\mathcal{A}_n|\right]^2\\
&=\int\limits_{(S^C\setminus L)\times (S^C\setminus L)}\left(\mathbb{P}\left(J\in X_n;\ x,y\in\mathcal{A}(J)\right)-\frac{1}{4}\right)d\nu(x)d\nu(y).\nonumber
\end{align}

\noindent Recall that the Chebyshev inequality says that for any $t>0$ and random variable $Y$ with expected value $\mathbb{E}[Y]$ and variance $\textnormal{Var}(Y)$, then
$$\mathbb{P}\left(|Y-\mathbb{E}[Y]|\geq t\right)\leq\frac{\textnormal{Var}(Y)}{t^2}.$$
Now we apply the Chebyshev inequality to the case that $Y=|\mathcal{A}_n|$ and $t=\delta\cdot\textnormal{Area}_C\left(S^C\right)$ where $\delta \in (0,\frac{1}{2})$ is arbitrary. Since $\mathbb{E}\left[|\mathcal{A}_n|\right]=\frac{1}{2}\textnormal{Area}_C\left(S^C\right)$, we have
\be\label{p-by-var}
 \mathbb{P}\left(J\in X_n;\ \left||\mathcal{A}_n(J)|-\frac{1}{2}\textnormal{Area}_C\left(S^C\right)\right|\geq \delta\cdot\textnormal{Area}_C\left(S^C\right)\right) \leq\frac{\textnormal{Var}\left(|\mathcal{A}_n|\right)}{\delta^2\left(\textnormal{Area}_C\left(S^C\right)\right)^2}.
\ene
The following lemma is motivated by \cite[Lemma 1]{BCP22}.
\begin{lemma}\label{l-exp-2}
For any $\delta>0$ and $n$ large enough, we have
$$\mathbb{P}\left(J\in X_n;\ \left|\frac{|\mathcal{A}_n|(J)}{\textnormal{Area}_C\left(S^C\right)}
-\frac{1}{2}\right|\geq \delta\right)\leq\delta.$$
\end{lemma}
\begin{proof}
By \eqref{p-by-var} it suffices to show that for $n$ large enough, 
\be \label{d-2-d-1}
\frac{\textnormal{Var}\left(|\mathcal{A}_n|\right)}{\delta^2\left(\textnormal{Area}_C\left(S^C\right)\right)^2}\leq \delta.
\ene

\noindent Recall that 
\[S^C=\mathfrak{S}_0^C \bigcup\{\text{$2n$ small triangles}\}\] where $$\mathfrak{S}_0^C=\{D_{ij};\ i\in\mathcal{I}_1,1\leq j\leq k_i\}\cup
\{D_i;\ i\in\mathcal{I}_2\}.$$
Now we split the product $(S^C\setminus L)\times (S^C\setminus L)$ as the following four parts:
\begin{align*}
U_1\overset{\text{def}}{=}\left\{(x,y)\in (S^C\setminus L)\times (S^C\setminus L);\begin{matrix} &\text{there exists a small triangle } \Delta\\ &\text{ with a sector domain }\Omega\text{ such that }\\ & \text{either }x\in\Delta,y\in\Omega\text{ or }y\in\Delta,x\in\Omega\end{matrix}\right\},
\end{align*}
\begin{align*}
U_2\overset{\text{def}}{=}\left\{(x,y)\in (S^C\setminus L)\times (S^C\setminus L);\begin{matrix}&\text{there exist two small triangles}\\ &\Delta_1\text{ and }\Delta_2\text{ such that $x\in\Delta_1,\ y\in\Delta_2$ and}\\ & \text{$\Delta_1$ and $\Delta_2$ share at least one common}\\&\text{ sector domain}\end{matrix}\right\},
\end{align*}
\begin{align*}
U_3\overset{\text{def}}{=}\left\{(x,y)\in (S^C\setminus L)\times (S^C\setminus L);\ x,y\in\Omega\text{ for some }\Omega\in\mathfrak{S}_0^C\right\},
\end{align*}
and
$$U_4\overset{\text{def}}{=}\left((S^C\setminus L)\times (S^C\setminus L)\right)\setminus\left(\bigcup\limits_{m=1}^3 U_m\right).$$
It is clear that
\[(S^C\setminus L)\times (S^C\setminus L)=\bigcup_{m=1}^4 U_m.\]

\noindent For any $(x,y)\in U_4\subset (S^C\setminus L)\times (S^C\setminus L)$, there are four cases:
\begin{enumerate}
\item $x\in\Omega_1,\ y\in\Omega_2$ for some $\Omega_1\neq\Omega_2\in\mathfrak{S}_0^C$;
\item $x\in\Omega,\ y\in\Delta$ for some $\Omega\in\mathfrak{S}_0^C$ and small triangle $\Delta$ such that $\Omega$ is not a sector domain of $\Delta$;
\item $x\in\Delta,\ y\in\Omega$ for some $\Omega\in\mathfrak{S}_0^C$ and small triangle $\Delta$ such that $\Omega$ is not a sector domain of $\Delta$;
\item $x\in\Delta_1,\ y\in\Delta_2$ for two small triangles $\Delta_1$ and $\Delta_2$ which do not share any common sector domain.
\end{enumerate}
For all the four cases of $U_4$ above, the two events for $x$ and $y$ are independent. So we have that for any $(x,y)\in U_4$,
\begin{align}\label{e-prob}
\mathbb{P}\left(J\in X_n;\ x,y\in\mathcal{A}(J)\right)=\mathbb{P}\left(J\in X_n;\ x\in\mathcal{A}(J)\right)\cdot \mathbb{P}\left(J\in X_n;\ y\in\mathcal{A}(J)\right)=\frac{1}{4}.
\end{align}
Thus, to prove \eqref{d-2-d-1}, from \eqref{e-var} it suffices to show that for $n$ large enough,
\be\label{e-b-vol}
(\nu\times\nu)\left(\bigcup\limits_{m=1}^3 U_m\right)\leq \delta^3\left(\textnormal{Area}_C\left(S^C\right)\right)^2
\ene
where $\nu\times \nu$ is the product measure on $S^C\times S^C$. The proof is split into the following three sublemmas.
\begin{subl}\label{sl-1} For $n$ large enough, we have
\[(\nu\times\nu)(U_1)\leq \frac{12n^2}{(\log n)^2}.\]
\end{subl}
\begin{proof}
For each small triangle $\Delta$, it follows by Lemma \ref{l-sch} that
\begin{align}\label{area-tri}
\textnormal{Area}_C(\Delta)\leq\textnormal{Area}_O(\Delta)=\pi-3<\frac{1}{6}.
\end{align}
From \eqref{e-sum} and Lemma \ref{l-decom} we know that for any $i\in\mathcal{I}_1$ and $1\leq j\leq k_i$, $$\textnormal{Area}_C\left(D_{ij}\right)\leq\frac{2d_i}{(\log n)^2}\leq\frac{12n}{(\log n)^2}.$$
By definition of $\mathcal{I}_2$ we also have that for any $i\in\mathcal{I}_2$,
$$\textnormal{Area}_C\left(D_i\right)\leq d_i\leq\frac{n}{(\log n)^2}.$$
In summary, for any $\Omega\in\mathfrak{S}_0^C$,
\begin{align}\label{area-do}
\textnormal{Area}_C(\Omega)\leq \frac{12n}{(\log n)^2}.
\end{align}
Since there are $2n$ small triangle and each small triangle has at most three sector domains, together with \eqref{area-tri} and \eqref{area-do}, we have
\begin{align}\label{e-area-1}
(\nu\times\nu)\left(U_1\right)&=
\sum\limits_{(\Delta,\Omega)}\textnormal{Area}_C(\Delta)\times
\textnormal{Area}_C(\Omega) \nonumber\\
&\leq2n\times 3\times\frac{1}{6}\times\frac{12n}{(\log n)^2}=\frac{12n^2}{(\log n)^2},\nonumber
\end{align}
where $(\Delta,\Omega)$ runs over all pairs of small triangle $\Delta$ and $\Omega\in\mathfrak{S}_0^C$ such that $\Omega$ is a sector domain of $\Delta$.
\end{proof} 

\begin{subl}\label{sl-2} For $n$ large enough, we have
\[(\nu\times\nu)(U_2)\leq \frac{3n^2}{(\log n)^2}.\]
\end{subl}
\begin{proof}
 For any $i\in\mathcal{I}_1$ and $1\leq j\leq k_i$, from Lemma \ref{l-decom} we know that the boundary $\partial D_{ij}$ contains at most $\frac{3d_i}{(\log n)^2}\leq\frac{18n}{(\log n)^2}$ short horocycle segments. For any $i\in\mathcal{I}_2$, the boundary $\partial D_i$ contains at most $d_i\leq\frac{n}{(\log n)^2}$ short horocycle segments. In summary, we deduce that for each $\Omega\in\mathfrak{S}_0^C$, it has at most $\frac{18n}{(\log n)^2}$ small triangles such that $\Omega$ is contained in the set of sector domains of each small triangle. This implies that for each small triangle $\Delta$, there are at most $\frac{54n}{(\log n)^2}$ small triangles which share at least one common sector domain with $\Delta$. Thus we have
\begin{align}\label{e-area-2}
(\nu\times\nu)(U_2)&=\sum\limits_{(\Delta_1,\Delta_2)}\textnormal{Area}_C(\Delta_1)\times
\textnormal{Area}_C(\Delta_2) \nonumber\\
&\leq2n\times\frac{54n}{(\log n)^2}\times\frac{1}{6}\times\frac{1}{6} =\frac{3n^2}{(\log n)^2},\nonumber
\end{align}
where $(\Delta_1,\Delta_2)$ runs over all pairs of small triangles sharing at least one common sector domain.
\end{proof}

\begin{subl}\label{sl-3} For $n$ large enough, we have
\[(\nu\times\nu)(U_3)\leq \frac{432c\cdot n^2}{\log n}.\]
\end{subl}
\begin{proof}
Since $\max\{|\mathcal{I}_1|,|\mathcal{I}_2|\}\leq |\mathcal{I}|\leq c\cdot\log n$, it follows by Lemma \ref{l-decom} that
\begin{align*}
\left|\mathfrak{S}_0^C\right|&=\sum\limits_{i\in\mathcal{I}_1}k_i+|\mathcal{I}_2|\\
&\leq 2(\log n)^2\times |\mathcal{I}_1|+c\cdot\log n\\
&\leq 3c\cdot(\log n)^3.
\end{align*}
Then combining with \eqref{area-do} we obtain 
\begin{align*}
(\nu\times\nu)(U_3)&=\sum\limits_{\Omega\in\mathfrak{S}_0^C}\textnormal{Area}_C(\Omega)^2\\
&\leq 3c\cdot(\log n)^3\times\left(\frac{12n}{(\log n)^2}\right)^2\\
&=\frac{432c\cdot n^2}{\log n}
\end{align*}
as desired.
\end{proof}

\noindent Now we return to prove \eqref{e-b-vol}. By \eqref{e-coarea} we know that 
  \begin{align}\label{e-coarea-1}
  \textnormal{Area}_C\left(S^C\right)\geq 2\pi(n-c\cdot\log n).
  \end{align}
Then it follows by the three sublemmas above that Equation \eqref{e-b-vol} clearly holds for large enough $n$. The proof is complete. 
\end{proof}

As a direct consequence of Lemma \ref{l-exp-2},
\begin{lemma}\label{l-exp-3}
For any $0<\delta<\frac{1}{2}$ and $n$ large enough,
$$\mathbb{E}\left[\left|\mathcal{MA}_n\right|\right]\geq \pi(1-2\delta)^2\cdot(n-c\cdot\log n).$$
\end{lemma}
\begin{proof}
From Lemma \ref{l-exp-2} we know that for any $0<\delta<\frac{1}{2}$ and $n$ large enough, 
$$\mathbb{P}\left(J\in X_n;\ \left|\mathcal{MA}_n\right|(J)> \left(\frac{1}{2}-\delta\right)\textnormal{Area}_C\left(S^C\right)\right)> 1-2\delta,$$
which implies
$$\mathbb{E}\left[\left|\mathcal{MA}_n\right|\right]\geq \frac{1}{2}(1-2\delta)^2\textnormal{Area}_C\left(S^C\right).$$

Then the conclusion follows by \eqref{e-coarea-1}.
\end{proof}

Now we are ready to prove Theorem \ref{mt-i}.

\begin{proof}[Proof of Theorem \ref{mt-i}]
For any $c>0$ and let $(\Gamma,\mathcal{O})\in\mathcal{F}^\star_n(c)$ be arbitrary. Then from Lemma \ref{l-exp-1} and Lemma \ref{l-exp-3} we have that for any $0<\delta<\frac{1}{2}$,
\begin{align*}
\limsup\limits_{n\to\infty}h\left(S^C\left(\Gamma,\mathcal{O}\right)\right)&\leq  \limsup\limits_{n\to\infty}\min_{J\in X_n}\frac{|\partial \mathcal{A}_n|(J)}{|\mathcal{MA}_n|(J)}\\
&\leq \limsup\limits_{n\to\infty}\frac{\mathbb{E}\left[\left|\partial\mathcal{A}_n\right|\right]}
{\mathbb{E}\left[\left|\mathcal{MA}_n\right|\right]}\\
&\leq\limsup\limits_{n\to\infty}\frac{\frac{3n}{2}+4c\cdot(\log n)^4}{(1-2\delta)^2\pi(n-c\cdot\log n)}\\
&=\frac{1}{(1-2\delta)^2}\cdot\frac{3}{2\pi}.
\end{align*}
Recall that \eqref{e-0} says that for $n$ large enough,
$$\textnormal{Prob}_n\left\{(\Gamma,\mathcal{O})\in\mathcal{F}_n^\star(c)\right\}\geq 1-\frac{2}{c}.$$

 Then the conclusion follows by letting $\delta\to 0$ and $c\to\infty$.
\end{proof}

\begin{rem*}
In the proof of Theorem \ref{mt-i}, a small triangle in $S^O$ is enclosed by three short horocycle segments each of which has length equal to $1$. If we replace each small triangle by a geodesic triangle enclosed by three geodesic segments of lengths equal to $2\log\left( \frac{1+\sqrt{5}}{2}\right)\sim 0.962$, then the proof of Theorem \ref{mt-i} actually can improve $\frac{3}{2\pi}$ to $\frac{3 }{\pi}\log \left( \frac{1+\sqrt{5}}{2}\right)$ in Theorem \ref{mt-i}. We are grateful to one referee for pointing out it to us. 
\end{rem*}

We enclose this work by the following direct consequence. First recall that
\begin{thm*} $($\cite[Theorem 4.1]{MR1677565}$)$
For $l$ sufficiently large, there is a constant $C(l)>0$ only depending on $l$ such that if $S^O$ is a punctured Riemann surface with cusps of length $\geq l$, then
$$\frac{1}{C(l)}h(S^O)\leq h(S^C)\leq C(l)h(S^O)$$ where $S^C$ is the conformal compactification of $S^O$.

Furthermore, $C(l)\to 1$ as $l\to\infty$.
\end{thm*}

\noindent As in \cite{MR1677565} we know that $l(\epsilon)\to\infty$ as $\epsilon\to 0$. So combining the theorem above, Theorem \ref{t-2} and Theorem \ref{mt-i} we also have
\begin{corollary}
Let $(\Gamma,\mathcal{O})$ be a random element of $\mathcal{F}^\star_n$. Then for any $\epsilon>0$,
$$\lim \limits_{n\to \infty}\textnormal{Prob}_n\left\{(\Gamma,\mathcal{O})\in \mathcal{F}^\star_n; \ h\left(S^O(\Gamma,\mathcal{O})\right)<\frac{3}{2\pi}+\epsilon\right\}=1.$$
\end{corollary}


\bibliographystyle{plain}
\bibliography{ref}
\end{document}